\documentclass[12pt,twoside]{article}

\usepackage{amsmath,amsthm,amscd,amssymb,tikz,pgfplots,bbm,graphicx,color}
\usepackage{mathrsfs}
\usepackage{enumerate}
\usepackage[normalem]{ulem}

\usepackage[margin=30mm]{geometry} 

\usepackage{dsfont}
\usepackage{hyperref}
\usepackage{natbib}

\newtheorem{assumption}[equation]{Assumption}
\newtheorem{theorem}{Theorem}
\newtheorem{conjecture}{Conjecture}

\newtheorem{lemma}[equation]{Lemma}
\newtheorem{proposition}{Proposition}

\newcommand{\p}{{\mathbb P}}

\newcommand{\R}{{\mathbb{R}}}	
	
\newcommand{\N}{{\mathbb N}}	
\newcommand{\V}{{\mathcal V}}		

\newcommand{\cS}{\mathcal{S}}
\newcommand{\C}{\mathcal C}

\newcommand{\A}{\mathcal A}

\newcommand{\1}{\mathds{1}}

\newcommand{\bp}{\mathbf{p}}
\newcommand{\by}{\mathbf{y}}
\newcommand{\bx}{\mathbf{x}}
\newcommand{\bz}{\mathbf{0}}
\newcommand{\bo}{\mathbf{1}}

\title{The critical curve of long-range percolation on oriented trees}

\author{Olivier Couronné \and Sandro Gallo \and Leonardo T. Rolla}

\date{}
\providecommand{\keywords}[1]{\textbf{{Keywords:}} #1}

\begin{document}

\maketitle

\begin{abstract}
We consider a long-range percolation model on homogeneous oriented trees with several lengths. We obtain the critical surface as the set of zeros of a specific polynomial with coefficients depending explicitly on the lengths and the degree of the tree. Restricting to the case of two lengths, we obtain new bounds on the critical parameters, monotonicity properties, as well as continuity of the critical curve, plus some partial results concerning its convexity. Our proofs rely on the study of the properties of the characteristic polynomial of the transition matrix of a multi-step Markov chain related to the model. 
\end{abstract}


\section{Introduction, definitions and results}

At each vertex of the rooted oriented $d$-ary tree, we add edges of lengths $k_1<k_2<\dots<k_m$, pointing to its $d^{k_i}$ descendants at $k_i$ generations below.
For $i=1,\dots,m$, we let edges of length $k_i$ be \emph{open} with probability $p_i\in(0,1)$, independently of all other edges.
An \emph{open path} is a sequence of vertices connected by open edges. For any pair of vertices $u$ and $v$, we write $u\rightarrow v$ if there exists an open path starting at $u$ and ending at $v$.
We denote $\mathcal C(v):=\{w:v\rightarrow w\}$ and say that \emph{there is percolation} if $\theta(\bp)>0$, where $\bp=(p_1,\dots,p_m)$ is the probability parametrising vector,
\[
\theta(\bp):=\p(|\mathcal C(o)|=\infty)
\]
and $o$ denotes the root of the tree.

Note that $\theta$ is non-decreasing in $\bp$.
Let $\cS_c \subseteq [0,1]^m$ denote the boundary between the set of parameters $\bp=(p_1,\ldots,p_m)$ for which there is percolation and the set of $\bp$ for which there is no percolation.
We refer to this region $\cS_c$ as the \emph{critical surface} (or, when $m=2$, the \emph{critical curve}).
The goal of this paper is to analyse $\cS_c$. 

\paragraph{A general result for the multi-edge model.}

Let $\Omega=\{0,1\}^{k_m}$ and denote its elements as $\bx=(x_1,\dots,x_{k_m})$.
Consider the following matrix doubly-indexed by $\{0,1\}^{k_m}\setminus\{0\}^{k_m}$:
\begin{equation}\label{def:matrixQ}
Q(\bx,\by)=
\left\{\begin{array}{ccc}
\prod_{j=1}^m (1-p_j)^{x_{k_m-k_j+1}} &\text{if}&\,y_{k_m}=0\,,\,\,y_i=x_{i+1},\,i=1,\dots,k_m-1,
\\
1-\prod_{j=1}^m (1-p_j)^{x_{k_m-k_j+1}} &\text{if}&\,y_{k_m}=1\,,\,\,y_i=x_{i+1},\,i=1,\dots,k_m-1.
\end{array}\right.
\end{equation} 
We denote by $P_Q$ the characteristic polynomial of $Q$ (which is parametrized by $\bp$).

\begin{theorem}\label{thm:multi-edges}
Suppose $\gcd \{k_1,\dots,k_m\} =1$.
Then
\begin{equation}
\nonumber
\cS_c\subseteq\{\bp\in[0,1/d^{k_1}]\times\dots\times[0,1/d^{k_m}]:P_Q(1/d)=0\}.
\end{equation}
\end{theorem}

The proof of Theorem~\ref{thm:multi-edges} is given in Section~\ref{sec:markov}.

Our next results will give further properties of the critical curve for the case $m=2$. But before that, let us illustrate Theorem~\ref{thm:multi-edges} with the simplest examples.

Whenever $m=2$, we will use $(l,k,p,q)$ instead of $(k_1,k_2,p_1,p_2)$, to lighten the notation.
In this case, we define the functions 
\begin{align*}
p_c(q):=\inf\{p\in(0,1):\theta(p,q)>0\}
\text{ and } 
q_c(p):=\inf\{q\in(0,1):\theta(p,q)>0\},
\end{align*}
leaving the dependency on $d$, $l$ and $k$ implicit.

For $l=1$ and $k=2$, the matrix $Q$, indexed by $\{01,10,11\}\times\{01,10,11\}$, becomes
\begin{equation}
Q=\left(\begin{array}{ccc}
0&1-p&p\\
q&0&0\\
0&(1-p)(1-q)&1-(1-p)(1-q)
\end{array}\right).
\end{equation}
According to Theorem~\ref{thm:multi-edges}, the critical curve is the set of $(p,q)\in[0,d^{-1}]\times[0,d^{-2}]$ satisfying $P_Q(1/d)=0$,
where $P_Q$ is the characteristic polynomial of $Q$.
So we get the following equation for the critical curve:
\[
\left\{(p,q)\in[0,1/d]\times[0,1/d^2]\,:\,p q^2 - \frac{pq}{d^2}-\frac{pq}{d} + \frac{p}{d^2} -q^2+ \frac{q}{d^2} + \frac{q}{d} - \frac{1}{d^3}=0\right\}.
\]
Solving this equation, we get, for $q < d^{-2}$,
\begin{equation}
\label{eq:pcexplicit}
p_c(q)=\frac{ q^2 - \frac{q}{d} - \frac{q}{d^2} + \frac{1}{d^3}}{q^2 - \frac{q}{d} - \frac{q}{d^2} +\frac{1}{d^2}}
,
\end{equation}
or, equivalently,
\[q_{c}(p)=\frac{1}{2d}+\frac{1}{2d^2}-\frac{\sqrt{(d-1)(3dp+d+p-1)}}{2d^2\sqrt{1-p}}.\]
When $d$ becomes large, with $p<1/d$, this value of $q_c$ is equivalent to the lower bound $(1-dp)/d^2$ of~\cite{de2023multirange}.
If we restrict to $p=q$, the critical parameter will depend only on $d$ and is given by the root of a degree-3 polynomial, which leads to an explicit formula.
For $l=1$ and $k=3$, the polynomial has degree four in $q$, and is thus solvable.

Using Theorem~\ref{thm:multi-edges}, one can make plots of the critical surface at any fixed $d,m,k_1,\dots,k_m$ that is computationally feasible.
For that, one needs to work with the characteristic polynomial of a ($2^{k_m}-1)\times( 2^{k_m}-1$) matrix whose entries depend on the vector $\bp$, and are computed using~\eqref{def:matrixQ}.
We refer to Figure~\ref{figure:k2} for $m=2$ with $k_1=1,k_2=2$.
For larger values of $k_2$, we can hardly distinguish between the critical curve and the lower bound. Figure~\ref{fig:3d} shows the case where $m=3$ and $k_1=1,k_2=2,k_3=3$.

\begin{figure}\label{figure:k2}
\begin{center}
\includegraphics[scale=0.5]{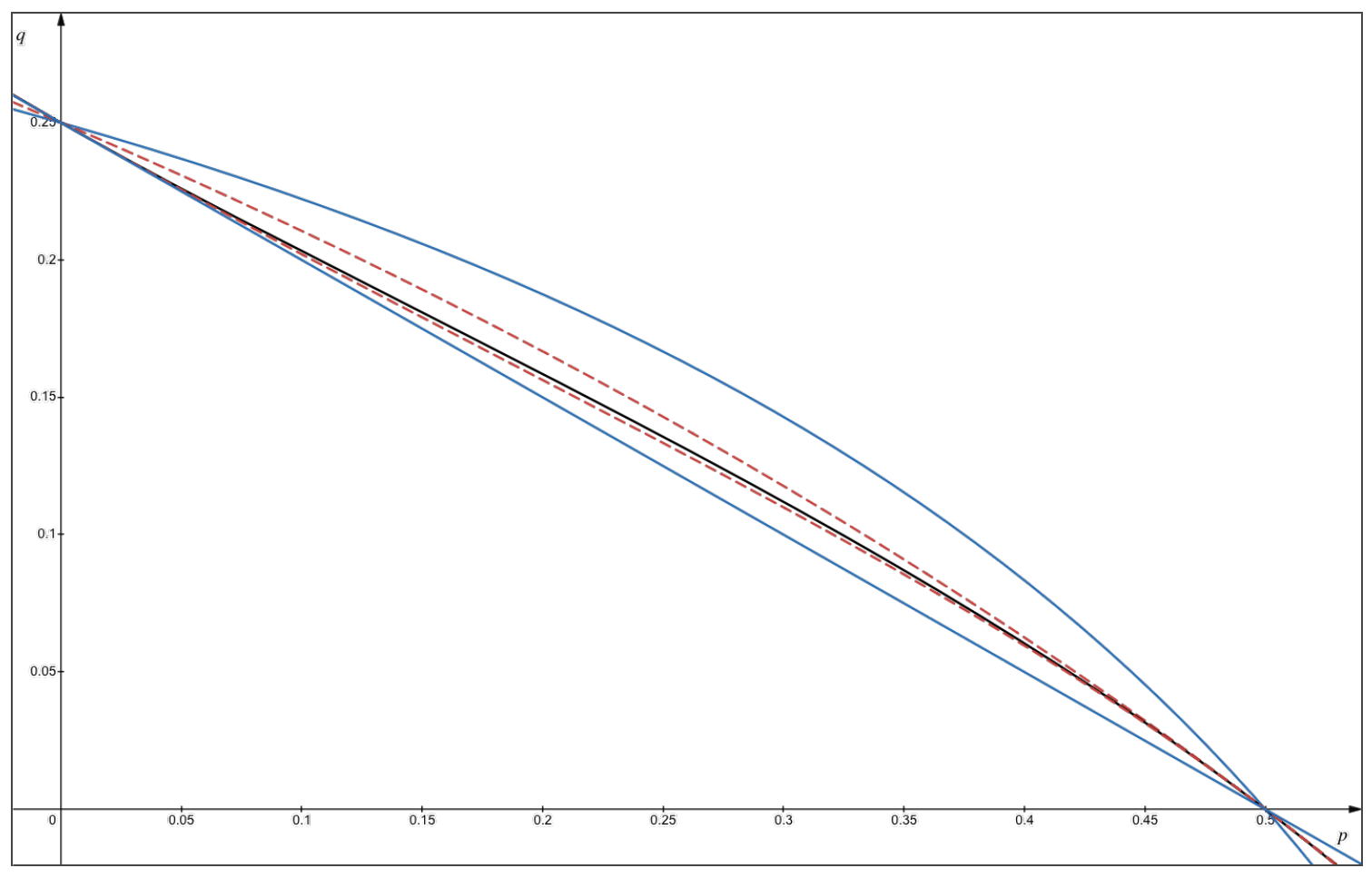}
\caption{In black the critical curve (the one of the middle) for $m=2$ with $l=1,k=2$. In blue the bounds of Proposition~\ref{prop:easybounds} and in red (dashed) the bounds of Theorem~\ref{thm:bounds}. }
\end{center}
\end{figure}

\paragraph{The two-edge model.}

We obtain further results when restricting to the case $m=2$.
Let us start with a simple observation.

\begin{proposition}
\label{prop:easybounds}
For every $p\in[0,1/d^l]$,
\[
\frac{1-pd^l}{d^k}
\le
q_c(p)
\le
\frac{1-pd^l}{(1-p)d^k}
.
\]
\end{proposition}
The lower bound is immediate by noticing that the percolation model can be embedded in a Galton-Watson tree with offspring expectation $pd^l+qd^k$.
For the upper bound, we consider a reduced model, whereby long edges originating at each given vertex can be open only when the corresponding short edges are closed.
On the one hand, this new model is dominated by the original model and, on the other hand, $\mathcal C(o)$ is in perfect correspondence with a Galton-Watson tree having offspring expectation $pd^l+(1-p)qd^k$.

The following theorem gives tighter bounds for the critical parameters.

\begin{theorem}
\label{thm:bounds}
Fix any $d\ge2$ and $1\le l<k<\infty$.
For every $p\in[0,1/d^l]$,
\begin{align*}
\frac{1-pd^l}{d^k(1-p^{k-l+1})}\le q_c(p)& \le \frac{1-pd^{l}}{d^{k}-pd^{l}}
.
\end{align*}
\end{theorem}
 
The proof of Theorem~\ref{thm:bounds} is given in Section~\ref{sec:double2}. 

Two interesting consequences of this theorem are the following. First, the lower bound implies that the critical curve lies strictly above the straight line that connects $(d^{-l},0)$ to $(0,d^{-k})$, for every $1\le l<k<\infty$ (as remarked in \cite{de2023multirange} when $l=1$). Second, the bounding curves are differentiable, and when $l=1$ they have equal slope $\frac{-d}{d^k-1}$ at $p=d^{-1}$, implying that $p\mapsto q_c(p)$ is also differentiable at $p=d^{-1}$, and has the same slope.

\begin{figure}\label{figure:k3}
\begin{center}
\includegraphics[scale=0.35]{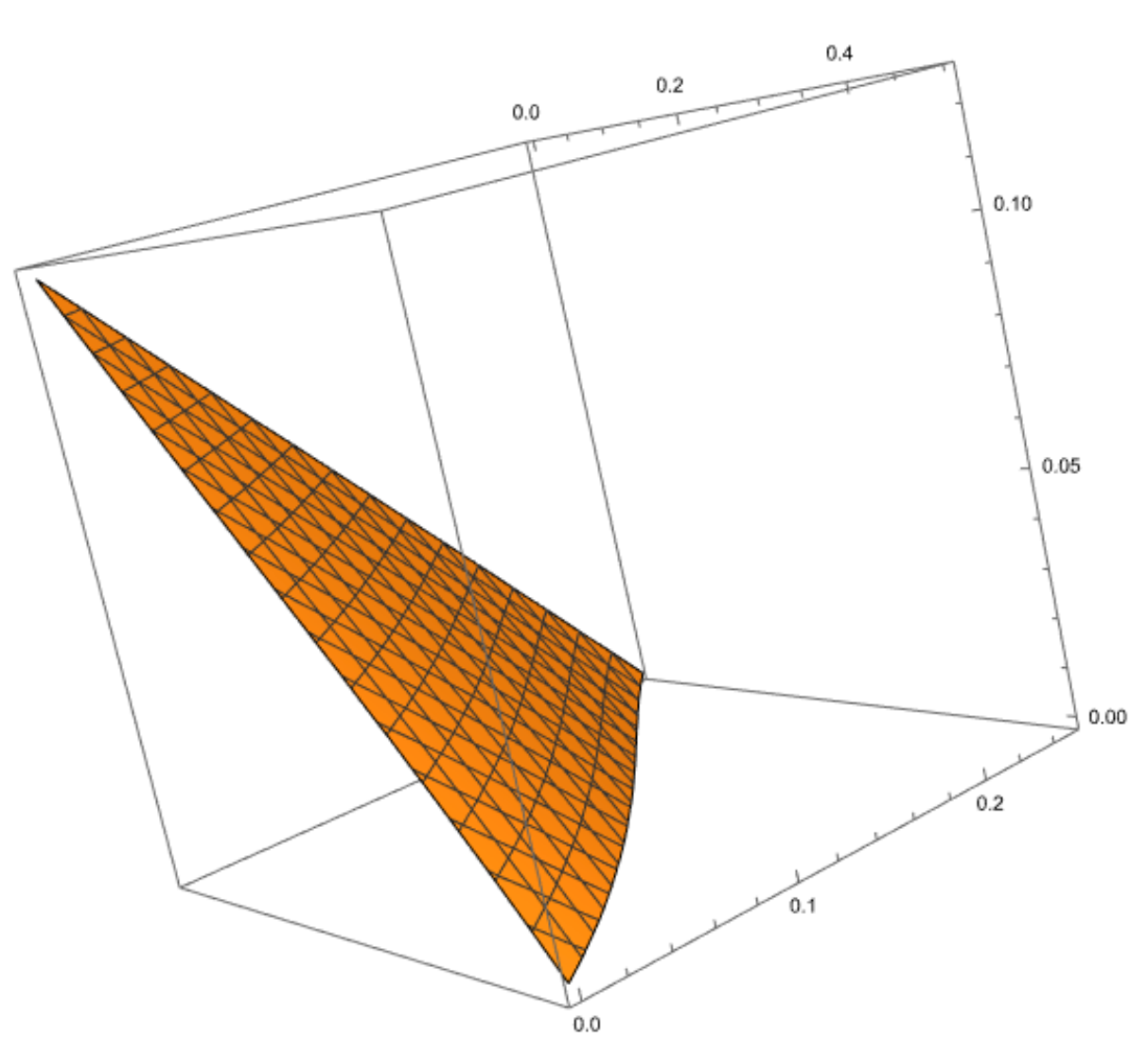}
\caption{The critical surface for $m=3$ with $k_1=1,k_2=2,k_3=3$. }
\label{fig:3d}
\end{center}
\end{figure}

Our last theorem states that the critical curve is continuous and strictly decreasing.

\begin{theorem}
\label{thm:curvemonot}
Fix any $d\ge2$ and $1\le l<k<\infty$.
If $ q' > q'' > 0 $ are such that $ p_c(q'') > 0 $, then $ p_c(q') < p_c(q'') $.
If $ p' > p'' > 0 $ are such that $ q_c(p'') > 0 $, then $ q_c(p') < q_c(p'') $. 
\end{theorem}
The proof of Theorem~\ref{thm:curvemonot} is given in Section~\ref{sec:monot}.

The first handmade sketches of the critical curve show it as a concave function. 
It turns out that this perhaps intuitive depiction is misleading.
The function $q \mapsto p_c(q)$ given by~\eqref{eq:pcexplicit} is convex near $q=\frac{1}{4}$; equivalently, $p \mapsto q_c(p)$ is convex near $p=0$, as can be seen on Figure~\ref{figure:k2}.
This seems to be the case for $l=1$ and every $k \geq 2$.

\begin{conjecture}
\label{conj:convexity}
For the case $m=2$ and $l=1$, there exists $\delta=\delta(k,d)>0$ such that $p \mapsto q_c(p)$ is not concave on $[0,\delta]$.
\end{conjecture}

We provide a justification for this conjecture in Section~\ref{sec:conjecture}, based on an assumption that we also conjecture to be true. 

\paragraph{Relation with previous works.} \cite{de2019monotonicity} considered the special case in which $m=2$ and $l=1$.
They proved that $q_c$ is continuous and strictly decreasing in $p$, and also strictly decreasing in $k$. Later on, \cite{de2023multirange} obtained the first two terms of the asymptotic expansion of $q_c$ as $k$ diverges, and in particular that the critical curve lies strictly above the line $q=\frac{1-pd}{d^k}$.

Let us first observe that, even in the broader context where $1 \le l < k < \infty$, the fact that $q_c=q_c(p,d,l,k)$ is strictly decreasing in $k$ follows from Proposition~\ref{prop:easybounds} above.

Now let us position our results referring to the critical surface.
Theorems~\ref{thm:bounds} and~\ref{thm:curvemonot} hold for every $l$ and $k$, without assuming that $l=1$.
Theorem~\ref{thm:bounds} implies that the critical curve lies strictly above the line connecting $(0,d^{-k})$ to $(d^{-l},0)$, and when $l=1$, that it is differentiable at $p=d^{-1}$ with slope $\frac{-d}{d^k-1}$, steeper than said line.

Let us also mention that our method of proof for Theorem~\ref{thm:multi-edges} does not involve percolation arguments as such used by~\cite{de2019monotonicity} and~\cite{de2023multirange}. Much more in the spirit of \cite{gallo2018frog} and \cite{gallo2023critical}, we transform the problem into that of studying the decay of the probability of finding long open paths along one fixed direction, and, in a second step, this probability is related to the absorption time of a Markov chain.
An alternative approach is to consider a multi-type branching process whose survival is equivalent to percolation in the long-range percolation model, which leads to the same polynomial equation for the critical surface.

The remaining of the paper is dedicated to proving the results (Section \ref{sec:proofs}) and justifying the conjecture (Section \ref{sec:conjecture}).

\section{Proofs of the results}\label{sec:proofs}

We start by defining a multi-step Markov chain whose asymptotic absorption rate is related to connectivity decay of the multi-range percolation model along a fixed branch.
We then use it to prove Theorem~\ref{thm:multi-edges}.
Finally, Theorems~\ref{thm:bounds} and~\ref{thm:curvemonot} are proved in separate subsections using the same construction.

\subsection{Identification of the critical surface}
\label{sec:markov}

Here we prove Theorem~\ref{thm:multi-edges} by analysing a multi-step Markov chain that will be used in later subsections.

Let $\V$ denote the set of vertices of the rooted oriented $d$-ary tree.
For $v,w\in\V$, we write $w \ge v$ if either $v= w$ or the oriented path from the root $o$ to $w$ passes by $v$, and in this case denote by $d(v,w)$ the length of the oriented path from $v$ to $w$.
Let $\V^{(n)}:=\{v\in\V:d(o,v)=n\}$, $\V^{\ge w}:=\{v\in\V:v\ge w\}$.
Let $\V^{\swarrow}$ consist of a fixed sequence of sites that form an infinite oriented path starting from the origin, and $\V^{\swarrow n}:=\{v\in\V^{\swarrow}:d(o,v)\ge n\}$.
(Here we consider the edges of the original tree rather than the long-range ones of our percolation model.)

Denote by $u_n=u_n(\bp)$ the quantity
\[
u_n(\bp) :=
\p( \C(o) \cap \V^{\swarrow n} \ne \emptyset)
.
\]
Observe that $u_n(\bp)$ is non-increasing in $n$ and non-decreasing in $\bp$.
We will study the asymptotic behaviour of $(u_n)_n$ through the following construction.

Fix $n\ge1$ for now.
We enumerate the first $n$ vertices in $\V^{\swarrow}$ backwards as $v_{n},\dots,v_1,v_0$, so that $v_n=o$, $v_0 \in \V^{(n)} \cap \V^\swarrow$ and $d(v_i,v_0)=i$ for $i=0,\dots,n$.
We now define a sequence of binary random variables 
\begin{equation}\label{eq:Y}
Y^{n}_{i}=\1\{v_{i} \rightarrow \V^{\swarrow n} \}, \ i=1,\dots,n,
\end{equation}
so that in particular
\[
\{Y^{n}_{n}=1\}=\{ \C(o) \cap \V^{\swarrow n} \ne \emptyset \}
\]
and therefore
\begin{equation}\label{eq:Yeu}
u_n=\p(Y^{n}_n=1).
\end{equation}

Notice that $(Y^{n}_{i})_{i=1,\dots,n}$ can be constructed as a $(k_m-1)$-step Markov chain started from $Y_{-k_m+1}=\dots=Y_0=1$ with transition probabilities 
\begin{equation}\label{eq:transitionY}
\p(Y_i=0|Y_{i-1}=y_{i-1},\dots,Y_{i-k_m}=y_{i-k_m})=\prod_{j=1,\dots,m}(1-p_j)^{y_{i-k_j}}
.
\end{equation}
These transition probabilities do not depend on $n$, so we removed $n$ from the superscript and we are no longer taking $n$ fixed.

This $(k_m-1)$-step Markov chain can be seen as a proper Markov chain $(X_j)_{j\ge 0}$ on $\{0,1\}^{k_m}$, by putting $X_j=(Y_{j-k_m+1},\dots,Y_{j})$ for $j=0,1,2,\dots$.
This chain starts from $X_{0} = \bo$ and has transition matrix given as follows: for $\bx,\by \in \{0,1\}^{k_m}$ 
\begin{equation}\label{eq:transitionmatrixQ}
\mathcal Q(\bx,\by)=
\left\{\begin{array}{ccc}
\prod_{j=1}^m (1-p_j)^{x_{k_m-k_j+1}} &\text{if}&\,y_{k_m}=0\,,\,\,y_i=x_{i+1},\,i=1,\dots,k_m-1,
\\
1-\prod_{j=1}^m (1-p_j)^{x_{k_m-k_j+1}} &\text{if}&\,y_{k_m}=1\,,\,\,y_i=x_{i+1},\,i=1,\dots,k_m-1.
\end{array}\right.
\end{equation} 
Observe that $\mathcal Q(\bz,\bz)=1$, so $\bz$ is an absorbing state. We therefore consider the matrix $Q$, sub-matrix of $\mathcal Q$ restricted to $\{0,1\}^{k_m}\setminus \{\bz\}$, which is precisely the one defined in~\eqref{def:matrixQ}.
Now let
\[
\rho_Q = \rho_Q(\bp) = \max\{|\lambda|:\lambda \in \mathbb{C} \text{ is an eigenvalue of } Q \}
.
\]
Since the entries of $Q$ vary continuously as functions of $\bp$, and $\rho_Q$ varies continuously with respect to the entries of the matrix, we automatically have that \emph{$\rho_Q$ varies continuously with $\bp$}.

The next lemma relates $\rho_Q$ to the asymptotic behavior of $u_n$.

\begin{lemma}
\label{lemma:utorho}
The $\lim u_n^{1/n}$ exists and equals $\rho_Q$.
\end{lemma}
\begin{proof}
For matrices doubly indexed by $\{0,1\}^{k_m} \setminus \{\bz\}$, consider the norm
\[
\|A\|_1 = \sup\{ |\mu A|_1 : |\mu|_1 = 1 \}
.
\]
By Gelfand's formula,
\[
\rho_Q = \lim_n (\|Q^n\|_1)^{1/n}
.
\]
We develop
\begin{align*}
\| Q^n \|_1
&=
\sup \bigg\{ \sum_{\by\ne\bz} \Big| \sum_{\bx\ne\bz} \mu(\bx) Q^n(\bx,\by) \Big| : \sum_{\bx\ne\bz} |\mu(\bx)|=1 \bigg\}
\\
& =
\sup \bigg\{ \sum_{\bx,\by\ne\bz} \mu(\bx) Q^n(\bx,\by) : \sum_{\bx\ne\bz} \mu(\bx)=1 \bigg\}
\\
& =
\sup \bigg\{ \p_\mu(X_n \ne \bz) : \sum_{\bx\ne\bz} \mu(\bx)=1 \bigg\}
\\
&=
\p_{\delta_{\bf1}}(X_n\ne{\bf0})
,
\end{align*}
where the second equality holds because the entries of $Q$ are non-negative, and the last equality holds by attractiveness of the model (larger configurations at time $0$ yield larger configurations at time $n$).

By definition of $X_n$, we have
\[
\| Q^n \|_1
=
\p_{\delta_\bo}(\cup_{i=n-k_m+1}^{n}\{Y_i=1\})
.
\]
Since $\p_{\delta_\bo}(Y_i=1) = u_i$, which is non-increasing in $i$, this yields
\[
u_{n-k_m+1}
\leq
\| Q^n \|_1
\leq
k_m
u_{n-k_m+1}
.
\]
Applying Gelfand's formula, we get
$\lim u_n^{1/n} = \rho_Q$, which concludes the proof.
\end{proof}

\begin{lemma}
\label{lemma:percu}
If $\rho_Q(\bp)<1/d$, then $\theta(\bp)=0$ and if $\rho_Q(\bp)>1/d$, then $\theta(\bp)>0$.
\end{lemma}

\begin{proof}
For $n\in \N$, consider the random set 
\[
\A_n:=\{w\in \V^{(n)}:\mathcal C(o)\cap \V^{\ge w}\ne\emptyset\}
\]
and observe that $d^n u_n \leq \mathbb E|\A_{n}| \leq d^{k_m} d^{n} u_{n}$.

We start by proving the first statement.
For each $n \in \mathbb{N}$, on the event that $|\mathcal C(o)|=\infty$,
we have $\A_n \ne \emptyset$.
Hence,
\[
\theta(\bp)
\leq
\p(\A_n \ne \emptyset)
\leq
d^{k_m} d^n u_n
.
\]
Assuming $\rho_Q(\bp)< \frac{1}{d}$, we can take $\alpha<\frac{1}{d}$ such that $\rho_Q(\bp)<\alpha$.
By Lemma~\ref{lemma:utorho}, we have $u_n < \alpha^n$ for all sufficiently large $n$.
Thus $d^{k_m} d^n u_n<d^{k_m}(d\alpha)^n\rightarrow0$, implying that $\theta(\bp)=0$. This concludes the proof of the first part.

We now prove the second part.
Assuming that $\rho_Q(\bp) > \frac{1}{d}$, we can take $\alpha>\frac{1}{d}$ such that $\rho_Q(\bp)>\alpha$. Since $\rho_Q(\bp)=\lim_n u_n^{1/n}$, we can take $n$ such that $u_n>\alpha^{n}$.
We will define a random subset $\mathcal T\subseteq\mathcal C(o)$ distributed as a Galton-Watson branching process whose offspring size has the same distribution as $|\A_n|$.
Since
$
\mathbb E|\A_{n}| \geq d^{n} u_{n} > \left(\alpha d\right)^{n}>1,
$
we have
$\mathbb P(|\mathcal T|=\infty)>0$, and therefore 
$\theta(\bp)>0$.

Let us define $\mathcal T$.
The $0^{\text{th}}$ generation of $\mathcal T$ is the root $o$.
The offspring of a vertex $w$ belonging to some generation of $\mathcal T$ is defined as follows.
Define the random set
\[
\A_n^{w}:=\{v\in\mathcal V^{w+n}:\mathcal C(w)\cap \mathcal V^{\ge v}\neq\emptyset\},
\]
where $\mathcal V^{w+n}:=\{x\in\mathcal V^{\ge w}:d(w,x)= n\}$.
For each $v\in \A_n^{w}$, select, among all vertices $u\in \mathcal V^{\ge v}$ such that $w \to u$, one that minimizes $d(w,u)$, to belong to the offspring of $w$.
With this construction, the size of the offspring of $w$ equals $|\A_n^w|$, which has the same distribution $|\A_n|$.
Moreover, the offspring of $w$ is conditionally independent of the offspring of other vertices in the same generation, showing that $\mathcal T$ is indeed distributed as a Galton-Watson branching process.
\end{proof}

The next lemma is a simple consequence of Lemma~\ref{lemma:percu}.
\begin{lemma}
\label{cor:rhocrit}
The critical surface of the multi-edge percolation model is contained in
\[
\{\bp\in[0,1/d^{k_1}]\times\dots\times[0,1/d^{k_m}]:\rho_Q(\bp)=1/d\}.
\]
\end{lemma}
\begin{proof}
If $\rho_Q(\bp)>1/q$, by continuity of $\rho_Q$ this inequality (and thus $\theta(\bp)>0$ by Lemma~\ref{lemma:percu}) holds in a neighbourhood of $\bp$, which implies that $\bp$ is away from the critical surface.
On the other hand, if $\rho_Q(\bp)<1/q$, by continuity of $\rho_Q$ this inequality (and thus $\theta(\bp)=0$ by Lemma~\ref{lemma:percu}) holds in a neighbourhood of $\bp$, which again implies that $\bp$ is away from the critical surface.
Therefore, on the critical surface, we can only have $\rho_Q = 1/d$.
\end{proof}

\begin{lemma}
\label{lemma:primitive}
Suppose $\gcd\{k_1,\dots,k_m\}=1$ and $\bp>\mathbf{0}$.
Then $Q$ is primitive. In particular, by the Perron-Frobenius Theorem, $\rho_Q$ is a real eigenvalue of $Q$.
\end{lemma}

\begin{proof}
The proof is somewhat trivial and only uses modular arithmetic.
Assume $\gcd\{k_1,\dots,k_m\}=1$ and let $M=k_1 \cdots k_m$.
We know that there is an integer linear combination $a_1 k_1 + \dots + a_m k_m = 1 $.
In particular, $a_1 k_1 + \dots + a_m k_m = 1 \mod M$. Taking $\tilde{a}_j = a_j \mod M$, we find positive coefficients such that $\tilde{a}_1 k_1 + \dots + \tilde{a}_m k_m = bM+1$ for some $b\in\N$.
So from any given position $i \in \N$, we can reach both $i+bM$ and $i+bM+1$.
From there, we can reach $i+2bM$, $i+2bM+1$ and $i+2bM+2$, and so on.
In the end, we can reach $i+k_mbM,\dots,i+k_m(bM+1)$.
Starting from any $\mathbf{x}\ne \mathbf{0}$, taking some $i \in \{1,\dots,k_m\}$ such that $x_i \ne 0$, after $k_m(bM+2)$ steps we can reach $\mathbf{x} = \mathbf{1}$ with positive probability, and from there we can reach any other $\mathbf{x}$ within another $k_m$ steps, also with positive probability, proving that $Q$ is primitive.
\end{proof}

We conclude this section with the proof of Theorem~\ref{thm:multi-edges}.

\begin{proof}
[\proofname\ of Theorem~\ref{thm:multi-edges}]
For $\bp$ on the critical surface, Lemma~\ref{cor:rhocrit} says that $\rho_Q(\bp)=1/d$.
It remains to justify that $1/d$ is actually an eigenvalue of $Q(\bp)$.
If $\bp > \bz$, then Lemma~\ref{lemma:primitive} says that $\rho_Q$ is an eigenvalue of $Q(\bp)$ and there is nothing left to prove.
So suppose $p_j=0$ for some $j$.
Note that, since $\bp$ is on the critical surface, we have $\theta>0$ in part of the neighbourhood of $\bp$.
Consider $\bp'=\bp+(\varepsilon,\dots,\varepsilon)$.
By monotonicity of $\theta$ we have $\theta(\bp')>0$, thus $\rho_Q(\bp')\geq 1/q$, and by the previous lemma the matrix $Q(\bp')$ has a real eigenvalue $\lambda_\varepsilon \geq 1/q$.
Letting $\varepsilon \downarrow 0$, we see that $Q(\bp)$ has a real eigenvalue $\lambda_0 \geq 1/d$.
Since $|\lambda| \leq \rho_Q$ for every eigenvalue $\lambda$, we conclude that $\lambda_0 = 1/d$, concluding the proof.
\end{proof}

\subsection{Recursive bounds}
\label{sec:double2}
Here we will prove Theorem~\ref{thm:bounds}. We are assuming that $m=2$ and denoting $k_1=l$, $k_2=k$, $p_1=p$ and $p_2=q$. 

Recall the process $(Y_n)_n$ defined in Section~\ref{sec:markov}, and $u_i=\p(Y_i=1)$.
For $i\ge k$
\begin{align}\label{eq:rec}
u_i=\p(Y_i=1)=&\p(Y_{i-k}=0,Y_{i-l}=1,Y_i=1)+\p(Y_{i-k}=1,Y_{i-l}=0,Y_i=1)+\nonumber\\&+\p(Y_{i-k}=1,Y_{i-l}=1,Y_i=1)\nonumber\\
=&\p(Y_{i-k}=0,Y_{i-l}=1)p+\p(Y_{i-k}=1,Y_{i-l}=0)q+\nonumber\\&+\p(Y_{i-k}=1,Y_{i-l}=1)(p+q-pq)\nonumber\\
=&\p(Y_{i-l}=1)p+\p(Y_{i-k}=1)q-\p(Y_{i-k}=Y_{i-l}=1)pq.
\end{align}

The third term prevents us from solving this recursive equation explicitly.
However, we can replace it by lower and upper bounds.
Observing that 
\[
\p(Y_{{i-l}}= Y_{{i-k}}=1)\le \p(Y_{{i-l}}=1)=u_{i-l}
,
\]
we obtain
\begin{align}
\label{eq:upper3}
u_i &\ge
p(1-q)u_{i-l}+qu_{i-k}
.
\end{align}
In the other direction, observing that
\[
\p(Y_{{i-l}}= Y_{{i-k}}=1)
=
\p(Y_{{i-l}}=1| Y_{{i-k}}=1)u_{i-k}
\ge
p^{k-l}u_{i-k}
,
\]
we obtain
\begin{align}
\label{eq:lower3}
u_i &\le
pu_{i-l}+q(1-p^{k-l+1})u_{i-k}
.
\end{align}

We can use these recursive equations to prove our bounds.

We start by proving the upper bound in Theorem~\ref{thm:bounds} using~\eqref{eq:upper3}.
Consider the recurrence
\begin{equation}
\label{eq:recurrence}
v_i = p(1-q) v_{i-l} + q v_{i-k}
.
\end{equation}
The polynomial associated to it is
\[
g(x)=x^k-p(1-q)x^{k-l}-q.
\]
By Descartes Rule of Signs, we have: (i) if $k$ is even, then $g$ has exactly two real roots, one positive and one negative, and (ii) if $k$ is odd then it has exactly one real root and it is positive.
In both cases, there is a unique positive root, that we denote $\rho$.

Consider the sequence $v_i = c \rho^i$, which satisfies~\eqref{eq:recurrence}.
The constant $c>0$ is taken so that $u_i \geq v_i$ for $i = 1,\dots, k$.
Now note that the same inequality can be extended for $i=k+1,k+2,\dots$ by using~\eqref{eq:upper3} and~\eqref{eq:recurrence} inductively.
In particular, $\lim_n u_n^{1/n} \ge \rho$.

Finally, since $g(x)<0$ for $0<x<\rho$ and $g(x)>0$ for $x>\rho$, we have the following equivalence:
\[
g({1}/{d}) < 0
\Longleftrightarrow
\rho > {1}/{d}.
\]

Thus, for every $(p,q)$ such that
\[
\frac{1}{d^k}-p(1-q)\frac{1}{d^{k-l}}-q<0,
\]
we have $\lim_n u_n^{1/n}>1/d$ and thus, using Lemmas~\ref{lemma:utorho} and~\ref{lemma:percu}, $ \theta(p,q) > 0 $.
In other words, if
\[
q>\frac{1-pd^l}{d^k-pd^l}
\]
then $\theta(p,q) > 0$.
This proves the upper bound of Theorem~\ref{thm:bounds}.

Finally, let us prove the lower bound using~\eqref{eq:lower3}.
Applying the exact same reasoning as above we get that, for every $(p,q)$ such that
\[
\frac{1}{d^k}-p\frac{1}{d^{k-l}}-q(1-p^{k-l+1})>0,
\]
we have $\lim_n u_n^{1/n} < 1/d$ and thus $\theta(p,q)=0$.
In other words, if
\[
q<\frac{1-pd^l}{d^k(1-p^{k-l+1})}
\]
then $\theta(p,q) = 0$.
This proves the lower bound of Theorem~\ref{thm:bounds}. 

\subsection{Strict monotonicity of the critical curve}\label{sec:monot}

Here we prove Theorem~\ref{thm:curvemonot}. Recall that we are assuming that $m=2$ and denoting $k_1=l$, $k_2=k$, $p_1=p$ and $p_2=q$.

\begin{lemma}
\label{lemma:sillybound}
For all $ q>0 $, we have $ p_c(q) < d^{-l} $.
For all $ p>0 $, we have $ q_c(p) < d^{-k} $.
\end{lemma}

\begin{proof}
We claim that
$\rho_Q \ge ( p^k + q^l - p^k q^l )^{1/lk}$.
Indeed, let $x_j:=\p(w\in\mathcal C(o))$ for any $w\in\V^{(jkl)}$.
This sequence is positive supermultiplicative, that is $x_{i+j} \ge x_ix_j$ for any $i,j \ge 0$.
Therefore $\lim_j x_j^{1/j}= \sup_j x_j^{1/j}$.
Now note that, among all possible open paths to connect two vertices that at distance $kl$ apart, one consists of $k$ short edges and another, independent of this one, consists of $l$ long edges.
Hence, $ x_{1} \geq p^k + q^l - p^k q^l $ and therefore $\lim_j x_j^{1/j} \geq p^k + q^l - p^k q^l$.
On the other hand,
$
u_{jkl} \geq x_j
$
and thus $\lim_n u_{n}^{1/n} = \lim_j u_{jkl}^{1/jkl} \geq \lim_j (x_j^{1/j})^{1/kl} \geq (p^k + q^l - p^k q^l)^{1/kl}$.
This proves the claim.

By Lemmas~\ref{lemma:utorho} and~\ref{lemma:percu}, $\lim_n u_{n}^{1/n} > 1/d$ implies that $\theta(p,q)>0$, concluding the proof.
\end{proof}

\begin{lemma}
\label{lemma:pczero}
For $q=0$, $\rho_Q = p^{1/l}$.
For $p=0$, $\rho_Q = q^{1/k}$.
\end{lemma}

\begin{proof}
For the first claim, it suffices to observe that
\[
\p_{\delta_\bo}(Y_n =1) = p^{\lceil n/l \rceil}
,
\]
thus
\[
\lim_n (\p_{\delta_\bo}(Y_n =1))^{1/n} = p^{1/l}
,
\]
which by Lemma~\ref{lemma:utorho} means $\rho_Q = p^{1/l}$.
The proof of the second statement is identical.
\end{proof}

We are ready to prove Theorem~\ref{thm:curvemonot}.
We can assume that $\gcd\{l,k\}=1$, for otherwise we can consider an equivalent model with parameters $l'=l/n$, $k'=k/n$ and $d'=d^n$, where $n=\gcd\{l,k\}$.
So we can use Lemma~\ref{lemma:primitive}.

\begin{proof}
[\proofname\ of Theorem~\ref{thm:curvemonot}]
Let $q'>q''>0$ and suppose $ p_c(q'')>0 $.
We want to prove that $ p_c(q') < p_c(q'') $.
We can assume that $ p_c(q')>0 $, otherwise the claim holds trivially.

By Lemma~\ref{cor:rhocrit}, $\rho_Q(p_c(q'),q')=1/d$ and $\rho_Q(p_c(q''),q'')=1/d$.
Suppose by contradiction that $ p_c(q') = p_c(q'')=p_c $.
By monotonicity of $\rho_Q$, we would have $\rho_Q(p_c,q)=1/d$ for every $q \in [q'',q']$.
By Lemma~\ref{lemma:primitive}, $1/d$ would be an eigenvalue of $Q_{p_c,q}$ for every $q \in [q'',q']$, that is, $P_Q(1/d;p_c,q)=0$ for every $q \in [q'',q']$.
Since $P_Q$ is a polynomial in all its three variables, we would deduce that $P_Q(1/d;p_c,q)=0$ for every $q\in \R$, and in particular $\rho_Q(p_c,q=0)=1/d$.
However, by Lemma~\ref{lemma:pczero} we have $\rho_Q(p_c,q=0)=p_c^{1/l}$ and by Lemma~\ref{lemma:sillybound} we have $p_c^{1/l}<1/d$, a contradiction.

This proves the first claim of the theorem. The second claim is analogous.
\end{proof}
 
\section{Convexity of the critical curve}\label{sec:conjecture}

In this section, 
we provide a justification for Conjecture~\ref{conj:convexity} based on the following ansatz regarding small values of $p$.
Let $d \geq 2$, $m=2$, $k_1=1$ and $k_2=k \geq 2$ be fixed.
\begin{assumption}
There exists a positive constant $C(k,d)$ such that,
for all $ p<\frac{1}{2d} $ and $ \frac{1}{2d^k} < q < \frac{2}{d^k} $, for all $n>k$,
\begin{equation}
\label{eq:assumption_conjecture}
\p(Y_{n-k}=1 | Y_{n-1}=1) \leq C(k,d) p.
\end{equation}\end{assumption}
We believe that this assumption is true.
Indeed, when $0 < p \ll q$, knowing that $Y_{n-1}=1$, most likely the position $n-1$ is reached by a chain of open edges that end with several long edges (of length $k$).
In such a configuration, we have $Y_{n-1-k}=1$ and, in order to also have $Y_{n-k}=1$ we need to open a short edge near the end of this chain of long edges, and this happens with probability proportional to $p$.

In the remainder of this section we justify Conjecture~\ref{conj:convexity} under the above assumption.

Recall~\eqref{eq:rec} and note that it can be rewritten as
\begin{equation}\label{eq:recursionci}
u_i=p u_{i-1} + q u_{i-k} - pq r_i u_{i-1},
\end{equation}
where
\begin{equation}
\nonumber
r_i := \p(Y_{i-k}=1 | Y_{i-1}=1)=\frac{\p(X_{i-1}(1)=X_{i-1}(k)=1|X_{i-1}\ne{\bf0})}{\p(X_{i-1}(k)=1|X_{i-1}\ne{\bf0})}.
\end{equation}

Recall the construction of the $(k-1)$-step Markov chain $(X_j)_j$ done in Section~\ref{sec:markov}.
The theory of quasi-stationary distributions for finite Markov chains (see \cite{collet2012quasi} for instance) tells us that
\[
\p(X_i = \mathbf{x} | X_i \neq \mathbf{0}) \to \lambda(\mathbf{x})
,
\]
where $\lambda$ is a probability measure supported on $\{0,1\}^{k}\setminus\{(0,\dots,0)\}$, called Yaglom limit.
So the conditional probabilities $ r_i $ have a limit, namely
\[
r_{p,q}
:=
\lim_i r_i =\frac{\lambda(\{\bx:x_1=x_k=1\})}{\lambda(\{\bx:x_k=1\})}
.
\]

Now notice that, for every $\varepsilon>0$, we have for sufficiently large $i$ that
\begin{equation}\label{eq:second_rec}
p u_{i-1} + q u_{i-k} - pq (r_{p,q}+\varepsilon) u_{i-1}\le u_i\le p u_{i-1} + q u_{i-k} - pq (r_{p,q}-\varepsilon) u_{i-1}.
\end{equation}
Just as we obtained bounds on $q_c=q_c(p,d,1,k)$ using recursive bounds \eqref{eq:upper3} and \ref{eq:lower3} on $u_i$, we  obtain, using the recursive bounds given by \eqref{eq:second_rec}, 
\[
\frac{1-pd}{d^k-p(r_{p,q_c}-\varepsilon)d}
\le
q_c
\le
\frac{1-pd}{d^k-p(r_{p,q_c}+\varepsilon)d}.
\]
Since these inequalities hold for every $\varepsilon>0$, we get
\[
q_c = \frac{1 - pd}{d^k - p r_{p,q_c} d}
.
\]

Since the critical curve is strictly above the line $q=\frac{1-pd}{d^k}$ (by Theorem \ref{thm:bounds}), which in turn has slope $-d^{1-k}$, Conjecture~\ref{conj:convexity} follows from the fact that
\[
q_c'(0) = -d^{1-k}
.
\] 
It remains to justify this fact.
For $ p $ small and $ r $ also small, we expand the previous expression for $q_c$ in terms of $ p $ and $ r$ as two independent variables, thus getting
\[
q_c = \frac{1 - pd}{d^k - p r d}
=
d^{-k}
+
\bigg[\frac{\partial}{\partial p}\frac{1 - pd}{d^k - p r d}\bigg]_{p=0,r=0}
p
+
\bigg[\frac{\partial}{\partial r}\frac{1 - pd}{d^k - p r d}\bigg]_{p=0,r=0}
r
+
o(p+r)
.
\]
The first partial derivative gives $ -d^{1-k} $ and the other vanishes.
Assumption~\eqref{eq:assumption_conjecture} implies that $ r_{p,q_c(p)} = O(p) $,
so the above expansion with $r = r_{p,q_c(p)}$ gives
$ q_c = d^{-k} - d^{1-k} p + o(p) $, showing that
$q_c'(0) = -d^{1-k}$ as claimed.

\section*{Acknowledgements}

This work has been supported by FAPESP through grants 2017/10555-0, 2023/13453-5, 2023/07507-5 and 2024/06341-9.



\bibliographystyle{jtbnew}
\bibliography{refperco}

\begin{thebibliography}{5}
\expandafter\ifx\csname natexlab\endcsname\relax\def\natexlab#1{#1}\fi
\expandafter\ifx\csname url\endcsname\relax
  \def\url#1{\texttt{#1}}\fi
\expandafter\ifx\csname urlprefix\endcsname\relax\def\urlprefix{URL }\fi
\providecommand{\selectlanguage}[1]{\relax}

\bibitem[{Collet \emph{et~al.}(2012)Collet, Mart{\'\i}nez \&
  San~Mart{\'\i}n}]{collet2012quasi}
\textsc{Collet, P., Mart{\'\i}nez, S. \& San~Mart{\'\i}n, J.} (2012).
\newblock \emph{Quasi-stationary distributions: Markov chains, diffusions and
  dynamical systems}.
\newblock Springer Science \& Business Media.

\bibitem[{de~Lima \emph{et~al.}(2019)de~Lima, Rolla \&
  Valesin}]{de2019monotonicity}
\textsc{de~Lima, B.~N., Rolla, L.~T. \& Valesin, D.} (2019).
\newblock Monotonicity and phase diagram for multirange percolation on oriented
  trees.
\newblock \emph{Random Structures \& Algorithms} \textbf{55}(1), 160--172.

\bibitem[{de~Lima \emph{et~al.}(2023)de~Lima, Szab{\'o} \&
  Valesin}]{de2023multirange}
\textsc{de~Lima, B.~N., Szab{\'o}, R. \& Valesin, D.} (2023).
\newblock Multirange percolation on oriented trees: Critical curve and limit
  behavior.
\newblock \emph{Random Structures \& Algorithms} \textbf{62}(2), 519--541.

\bibitem[{Gallo \& Pena(2023)}]{gallo2023critical}
\textsc{Gallo, S. \& Pena, C.} (2023).
\newblock Critical parameter of the frog model on homogeneous trees with
  geometric lifetime.
\newblock \emph{Journal of Statistical Physics} \textbf{190}(2), 34.

\bibitem[{Gallo \& Rodriguez(2018)}]{gallo2018frog}
\textsc{Gallo, S. \& Rodriguez, P.~M.} (2018).
\newblock Frog models on trees through renewal theory.
\newblock \emph{Journal of Applied Probability} \textbf{55}(3), 887--899.

\end{thebibliography}

{
Olivier Couronné \\
\small{Modal’X, Paris Nanterre University, France}\\
\small{\texttt{olivier.couronne@parisnanterre.fr} }
\vspace{0,5cm}
\\
Sandro Gallo \\ 
\small{Departament of Statistics, Federal University of São Carlos, Brazil}\\
\small{\texttt{sandro.gallo@ufscar.br} }
\vspace{0,5cm}
\\
Leonardo T. Rolla  \\
\small{Institute of Mathematics, Statistics and Computer Science, University of São Paulo, Brazil}\\
\small{\texttt{leonardo.rolla@gmail.com}} }

\end{document}